\documentclass{amsart}
\usepackage{amsmath}
\usepackage{amssymb}
\usepackage{amsthm}
\usepackage{amsfonts}
\usepackage{xcolor}
\usepackage[backend=biber]{biblatex}
\usepackage{hyperref}
\bibliography{main.bbl}

\newtheorem{theorem}{Theorem}[section]
\newtheorem{lemma}[theorem]{Lemma}
\newtheorem{corollary}[theorem]{Corollary}

\theoremstyle{definition}
\newtheorem{definition}[theorem]{Definition}

\theoremstyle{remark}
\newtheorem{remark}[theorem]{Remark}

\DeclareMathOperator{\diam}{\mathrm{diam}}

\newcommand{\floor}[1]{\left\lfloor #1 \right\rfloor} 
 
\newcommand{\br}[1]{\left\{ #1 \right\}}
\newcommand{\abs}[1]{\left| #1 \right|}
\newcommand{\pr}[1]{\left( #1 \right)}

\newcommand{\N}{\mathbb{N}}

\newcommand{\bP}{\boldsymbol{\Pi}}

\author[B. Mance]{Bill Mance}
\address[B. Mance]{Uniwersytet im. Adama Mickiewicza w Poznaniu, Collegium Mathematicum, ul. Umultowska 87, 61-614 Pozna\'{n}, Poland}
\email{william.mance@amu.edu.pl}

\author{Jakub Tomaszewski}
\address[J. Tomaszewski]{
	AGH University of Krakow, Faculty of Applied Mathematics,
    al.\ Mickiewicza 30,
    30-059 Krak\'ow, Poland.
}
\email{tomaszew@agh.edu.pl}

\title[Hausdorff dimension of difference sets between \textit{TpcN} and \textit{Normal}]{On the Hausdorff dimension of the difference sets between typical and normal numbers}

\begin{document}
\begin{abstract}
    We study the Hausdorff dimension of the difference sets between the set of typical numbers, defined via the Erd\H{o}s--R\'enyi law governing the longest run of digits, and the set of numbers normal in base $2$. Although both properties hold for Lebesgue-almost every real number, neither implies the other. The resulting difference sets $\textit{TpcN}\setminus\textit{Normal}$ and $\textit{Normal}\setminus\textit{TpcN}$ are $D_2(\bP_3^0)$-complete in the Borel hierarchy \cite{my_paper}. We show that this logical complexity is matched geometrically: both difference sets have full Hausdorff dimension. Unlike the classical Besicovitch-type sets governing single-digit frequencies, normality requires the correct frequency of every finite block simultaneously, and the explicit Moran-set constructions that succeed for simple normality break down under this requirement. We instead establish the dimension via a counting argument bypassing the need for an explicit construction.
\end{abstract}

\maketitle

\section{Introduction}

Let $\tilde{X}=\{0,1\}^\mathbb{N}$ be the space of infinite binary sequences with the product topology induced from the discrete topology on $\{0,1\}$. A compatible metric is defined for $\tilde{x}, \tilde{x}'\in \tilde{X}$ as follows:
\begin{equation*}
    d(\tilde{x}, \tilde{x}'):=\begin{cases}
             0, \;\text{if} \:\tilde{x}= \tilde{x}', \\
             {2^{-k}},\; \text{where} \:k=\min\{j\in\mathbb{N}\colon x_j\neq x'_j\}, \;\text{if otherwise.}
        \end{cases}
\end{equation*}
Let $\tilde{x}=(x_n)_{n=1}^\infty$ be a binary sequence. It is well-known that any such sequence can be associated to a number $x\in[0, 1]$ by a formula
\begin{equation*}
    x=\sum_{n=1}^\infty\frac{x_n}{2^n}.
\end{equation*}
Then, $\tilde{x}$ is the \emph{binary expansion} of $x$.
\begin{remark}
    It is common to consider $\tilde{x}$ as an infinite binary word instead of an infinite sequence in a product space $\{0, 1\}^\mathbb{N}$. In this view, we will freely use throughout this paper the notation of the standard concatenation of finite blocks, meaning that if $A=(a_1,\dots, a_n)\in\{0, 1\}^n, B=(b_1,\dots, b_m)\in\{0, 1\}^m$ for $n, m\in\mathbb{N}$, then $AB=(a_1,\dots, a_n, b_1,\dots, b_m)\in\{0, 1\}^{n+m}$. Similarly, we denote these concatenations by
    $$AB=a_1,\dots, a_n, b_1,\dots, b_m,$$ i.e. we write sequences of digits as strings. We extend this notation to infinite words as infinite concatenations of finite blocks. 
\end{remark}

\begin{remark}
    In this paper, we will always use the $\log$ function to denote the base $2$ logarithm while the $\ln$ function will be used for the natural logarithm.
\end{remark}

The study of typical numbers has its origin in probability theory rather than in metric number theory. In 1970, Erd\H{o}s and R\'enyi proved in \cite{erdos_renyi} what they called a \emph{new law of large numbers}, describing the length of the longest run of successes in a sequence of independent coin tosses. Writing $G_n$ for the cumulative outcome after $n$ tosses, where a success contributes $+1$ and a failure contributes $-1$ to $G_n$, their result (in the special case of a fair coin) states that
\begin{equation*}
    P\left(\lim_{N\to\infty}\max_{0\leq n\leq N-[\log N]}\frac{G_{n+[\log N]}-G_n}{[\log N]}=1\right)=1,
\end{equation*}
where $[x]$ denotes the integer part of $x$. In words, if we record the outcome of the $n$-th toss as $x_n=1$ for a success and $x_n=0$ for a failure, then \emph{typically} (i.e. almost surely) the length of the longest run of consecutive successes among the first $N$ tosses is asymptotic to $\log N$. Identifying a sequence of coin tosses with a binary sequence $\tilde x=(x_n)_{n=1}^\infty$ in the natural way, we land in the world of binary expansions of numbers in $[0, 1]$.

In any binary expansion $\tilde{x}$ of a number $x\in[0, 1]$ we can look for different regularities. One approach is to analyze the behavior of the so-called \emph{run-length} function
\begin{equation*}
    R_n(\tilde{x})=\max\{j\colon x_{i+1}=...=x_{i+j}=1\; \text{for some}\; 0\leq i\leq n-j\},
\end{equation*}
which returns the length of the longest run of consecutive ones among the first $n$ elements of $\tilde{x}$. This approach directly links the resulting observations with the assertions in \cite{erdos_renyi}.
\begin{definition}
    We say that a number $x\in[0, 1]$ is \emph{typical} if its binary expansion $\tilde{x}$ satisfies $\lim_{n\to\infty}\frac{R_n(\tilde{x})}{\log n}=1$.
\end{definition}
Hence, the special case of the Erd\H{o}s--R\'enyi law relevant to us can now be stated as follows.
\begin{theorem}[{\cite[Theorem 1]{erdos_renyi}}]
    The set of typical numbers $\textit{TpcN}\subset[0, 1]$ is of full Lebesgue measure.
\end{theorem}
\begin{remark}
    Note that the limit $\lim_{n\to\infty}\frac{R_n(\tilde{x})}{\log n}$ may not be equal to 1 or even exist. There have been a number of studies on sets of numbers to which this remark applies (see e.g. \cite{lim_normal_liminfsup_typical, liminfsup_normal_lim_typical}).
\end{remark}

Another natural approach to regularities of the binary expansion $\tilde{x}$ is based on the frequency with which finite blocks occur along it. For a finite or infinite binary string $\tilde{y}$, let $N(\tilde{y}, \omega, n)$ denote the number of occurrences of a finite block $\omega\in\{0, 1\}^k$ among the first $n$ elements of $\tilde{y}$.

\begin{definition}
    We say that a number $x\in[0, 1]$, whose binary expansion is $\tilde{x},$ is \emph{normal} (in base $2$) if for all $k\in\mathbb{N}$ and any finite block $\omega\in\{0, 1\}^k$ of length $k$ we have $\lim_{n\to\infty} \frac{N(\tilde{x}, \omega, n)}{n}=\frac{1}{2^k}$.
\end{definition}

Normal numbers were introduced by Borel \cite{borel}, who proved that Lebesgue-almost every real number is normal in every integer base. Thus, if $\mathcal{L}$ denotes Lebesgue measure on $[0,1]$, then the set of normal numbers $\textit{Normal}\subset [0, 1]$ satisfies
\begin{equation*}
    \mathcal{L}(\textit{Normal})=1.
\end{equation*}
It is natural to compare normality with typicality. A simple probabilistic heuristic suggests that normality should imply typicality. Indeed, in a random binary word of length $2^k$, the expected number of occurrences of the block $1^k$ is of order
\begin{equation*}
    2^k\cdot 2^{-k}=1.
\end{equation*}
One therefore expects a run of $k$ consecutive ones to appear around position $2^k$, which corresponds to
\begin{equation*}
    R_n(x)\approx\log n.
\end{equation*}
This heuristic is consistent with normality, since a normal number has the correct limiting frequency of every fixed finite block. There is, however, an important distinction: normality controls each fixed block separately, whereas typicality concerns runs whose length grows with the observation scale. Thus normality does not provide the uniform control over growing block lengths that the heuristic would require. In fact, neither property implies the other: in \cite{my_paper} it was shown by explicit construction that there exist normal numbers that are not typical, as well as typical numbers that are not normal.
\begin{remark}
    To avoid the ambiguity of binary expansions at dyadic rationals, throughout this paper we the expansion which is eventually equal to $0$. Hence we identify a number in $[0,1]$ with its chosen binary expansion. It will become clear that this approach has no impact on our considerations, as dyadic rationals are excluded from our constructions.
\end{remark}
Clearly, $\mathcal{L}(\textit{TpcN}\cap\textit{Normal})=1$, as both \textit{TpcN} and \textit{Normal} have full Lebesgue measure on $[0, 1]$.
Despite this measure-theoretic agreement, the distinction between the two notions is substantial from the point of view of the Borel hierarchy, a descriptive set-theoretical method of quantifying logical complexity (see e.g. \cite{miller_hierarchies} for more details). Both $\textit{TpcN}$ and $\textit{Normal}$ are $\bP_3^0$-complete, while the difference sets $\textit{TpcN}\setminus\textit{Normal}$ and $\textit{Normal}\setminus\textit{TpcN}$ are $D_2(\bP_3^0)$-complete, i.e. they have the highest logical complexity available to them \cite{ki_linton,my_paper}.
This raises a natural geometric question: how large are these difference sets in terms of Hausdorff dimension? Descriptive complexity alone gives no information about geometric size; sets positioned at relatively high Borel difference classes may have Hausdorff dimension zero. Our main result shows that this does not occur here: both difference sets have full Hausdorff dimension.

A natural point of comparison is provided by the classical Besicovitch digit-frequency sets. Let $x\in[0,1]$ and define
\begin{equation*}
    S_N(x)=\sum_{n=1}^N x_n.
\end{equation*}
For $\alpha\in[0,1]$, let
\begin{equation*}
    B(\alpha)=\br{x\in[0,1]\colon\lim_{n\to\infty}\frac{S_n(x)}{n}=\alpha}.
\end{equation*}
Thus $B(\alpha)$ consists of those binary expansions in which the digit $1$ occurs with limiting frequency $\alpha$. Sets of this type originate in the work of Besicovitch \cite{Besicovitch}; sets defined by prescribed digit frequencies in other integer bases are commonly referred to as \emph{Besicovitch--Eggleston sets}, following the subsequent work of Eggleston. Here we make no attempt to summarize the relevant literature, but it is worth mentioning a classical result of Colebrook \cite{Colebrook} that improves on the result of Eggleston.

Let
\begin{equation*}
    H(x)=-x\ln x-(1-x)\ln (1-x),
\end{equation*}
where $0\log 0=0$. Then
\begin{equation*}
    \max_{x\in[0,1]}H(x)=H\left(\frac{1}{2}\right)=\ln 2.
\end{equation*}
For $\beta\in[0,\infty]$, define
\begin{equation*}
    E(\beta)=\br{x\in[0,1]\colon\lim_{n\to\infty}\frac{R_n(x)}{\log n}=\beta}.
\end{equation*}
In particular,
\begin{equation*}
    E(1)=\textit{TpcN}.
\end{equation*}
The following theorem of Zhang and Peng describes the interaction between prescribed single-digit frequencies and run-length behavior.
\begin{theorem}[{\cite[Theorem 1.1]{lim_normal_lim_typical}}]
\label{BesicovitchTypical}
For every $\alpha\in[0,1]$ and $\beta\in[0,\infty]$,
\begin{equation*}
    \dim_H\bigl(B(\alpha)\cap E(\beta)\bigr)=\frac{H(\alpha)}{\ln 2}.
\end{equation*}
\end{theorem}
Theorem~\ref{BesicovitchTypical} provides a natural starting point for the present problem. The set $B(\alpha)$ imposes a single digit-frequency condition, while $E(\beta)$ prescribes the asymptotic behavior of the longest run. In the arguments leading to Theorem~\ref{BesicovitchTypical}, these conditions can be controlled along a suitably chosen sequence of scales. The relevant quantities are sufficiently stable between consecutive scales: the digit sum $S_n(x)$ can change only by the number of newly added digits, while the run-length function is monotone and can likewise be controlled once the scales are chosen appropriately. Note that in this approach we are given considerable freedom, as $\lim_{x\to\infty}(\log x)'=\lim_{x\to\infty}1/x=0$ and hence the sequence $\frac{R_n(x)}{\log n}$ cannot fluctuate substantially between a wide range of scales. This makes it possible to construct sets with the required behavior by prescribing suitable finite blocks at successive stages.

Normality introduces a substantially different difficulty. Instead of controlling a single frequency, one must require
\begin{equation*}
    \frac{N(x,\omega,n)}{n}\longrightarrow 2^{-|\omega|}
\end{equation*}
for every finite binary block $\omega$. At a finite stage one may, of course, impose these conditions only for blocks up to some length $k$, with a prescribed error tolerance, and then let $k\to\infty$ and the tolerance tend to zero along the construction. The problem is that the number of conditions being imposed now grows with the stage, and the estimates used for a single digit frequency do not provide sufficiently uniform control over this growing family of overlapping block statistics. In particular, occurrences of a block may cross the boundaries between concatenated pieces, and a construction must ensure that corrections introduced at later stages do not destroy the frequency estimates already obtained at earlier ones.

This creates a serious obstacle for a direct adaptation of the classical Moran-set approach. In the single-frequency setting, one can explicitly choose large families of finite blocks whose proportion of ones is close to a prescribed value and then concatenate them. The construction is robust because altering a negligible proportion of digits has a negligible effect on the only statistic being controlled. For normality, however, a block chosen at a given stage must satisfy many frequency conditions simultaneously. As the admissible block lengths increase, one must control an increasing collection of overlapping patterns, the errors produced at concatenation points, and the influence of all preceding stages. At the same time, a Hausdorff dimension lower bound requires sufficiently many choices to remain available at every stage. Thus it is not enough merely to exhibit some blocks with approximately correct frequencies: one needs quantitative information showing that there are \emph{many} such blocks.

This is precisely where the previous explicit construction becomes inefficient. Attempting to prescribe the admissible blocks individually leads to increasingly complicated bookkeeping, while the dimension estimate depends not on the identity of the blocks but on their abundance. This suggests reversing the point of view: instead of constructing good blocks one by one, estimate directly how many good blocks there are.

Our argument follows this approach. At stage $k$, we regard binary words of a suitable length as outcomes of independent fair coin tosses and estimate the proportion that fail to have approximately the expected number of occurrences of some block $\omega\in\{0,1\}^k$. Although the indicators of occurrences of $\omega$ may overlap, their starting positions can be separated into residue classes modulo $k$, within each of which the indicators are independent. Chernoff--Hoeffding estimates, followed by a union bound over the $2^k$ possible blocks $\omega$, show that the proportion of words failing the $(\epsilon_k,k)$-normality condition can be made arbitrarily small.

Consequently, at every stage $k$ a proportion at least $1-\delta_k$ of all binary words remains admissible, so the number of available words is still exponential in their length. We use these large admissible families as the branches of a homogeneous Moran construction, inserting short runs of ones between successive blocks to force failure of typicality. Lemma~\ref{epsilonkmnormality} transfers the stage-$k$ estimates to every fixed shorter block length, with errors tending to zero, and the hotspot lemma then yields normality of every point in the resulting set. Since the inserted runs have negligible relative length and the admissible families retain asymptotically full exponential size, the Moran set has full Hausdorff dimension. Thus, the counting argument avoids having to select or describe the normal blocks individually while preserving the branching needed for the dimension estimate.

This counting construction establishes the nontrivial inclusion needed for the set $\textit{Normal}\setminus\textit{TpcN}$, while the full Hausdorff dimension of $\textit{TpcN}\setminus\textit{Normal}$ follows from Theorem~\ref{BesicovitchTypical}. Together, these arguments prove the following theorem.

\begin{theorem}\label{Main}
    The Hausdorff dimension
\begin{equation*}
\dim_H(\textit{TpcN}\setminus\textit{Normal})=\dim_H(\textit{Normal}\setminus \textit{TpcN})=1.
\end{equation*}
\end{theorem}

Thus, the classical and the new laws of large numbers, despite holding simultaneously almost surely, are independent both logically and geometrically.

\section{Preliminary definitions and notions}
We will need to define the notion of $(\epsilon,k)$-normality, which is a central concept in the study of normal numbers.
\begin{definition}\label{eps_block}
    Let $k \in \N$ and $\epsilon>0$. A finite binary block of length $\ell$ is {\it $(\epsilon, k)$-normal} if the number of occurrences of every binary block of length $k$ is strictly between $(2^{-k}-\epsilon)\ell$ and $(2^{-k}+\epsilon)\ell$.
\end{definition}
Slight variants of Definition~\ref{eps_block} appear in the literature. See the book of Bugeaud \cite{Bugeaud} for information on the history of $(\epsilon,k)$-normality.
We will need the following basic lemma about $(\epsilon,k)$-normality.
\begin{lemma}\label{epsilonkmnormality}
    Let $k, m \in \N$ and $\epsilon>0$. Suppose that a finite block of digits $B$ is $(\epsilon,k+m)$-normal. Then, $B$ is $\pr{2^m\epsilon+\frac{m}{|B|},k}$-normal.
\end{lemma}
\begin{proof}
    Let $\omega \in \{0,1\}^k$. We note that
    $$
    \sum_{\phi \in \{0,1\}^m}N(B, \omega \phi, |B|) \leq N(B, \omega, |B|) \leq \sum_{\phi \in \{0,1\}^m}N(B, \omega \phi, |B|)+m,
    $$
    where the extra $m$ term in the upper bound comes from possible occurrences of $\omega$ in the last $m$ places of $B$ that may not be counted by blocks of the form $\omega \phi$.
    Since $B$ is $(\epsilon,k+m)$-normal, we have that for all $\phi \in \{0,1\}^m$
    $$
    |B|(2^{-(k+m)}-\epsilon)<N(B, \omega \phi, |B|)<|B|(2^{-(k+m)}+\epsilon).
    $$
    Summing over all $2^m$ possible values of $\phi$, we see that
    $$
    |B|(2^{-k}-2^m\epsilon)<N(B, \omega, |B|)<|B|(2^{-k}+2^m\epsilon)+m = |B|\pr{2^{-k}+\pr{2^m\epsilon+\frac{m}{|B|}}}.
    $$
    
\end{proof}
We will use a version of the classical hotspot lemma which may be found in \cite[Theorem 4.6]{Bugeaud}
\begin{theorem}[Hotspot lemma]\label{constant_C}
The number $x\in[0, 1]$ is normal if and only if there exists a positive constant $C$ such that for all $k\in\mathbb{N}$ and finite blocks $\omega\in\{0, 1\}^k$ of length $k$ we have
\begin{equation*}
    \limsup_{n\to\infty}\frac{N(x, \omega, n)}{n}\leq \frac{C}{2^k}.
\end{equation*}
\end{theorem}

Before we do so, let us recall the definition of the Hausdorff dimension. For further details, see \cite[Chapter 2]{Falconer}.
\begin{definition}
 Let $(X,d)$ be a metric space and $S\subset X$. 
 We set 
 $$
 \mathcal{H}^s(S)=\lim_{\delta\to 0}\inf\br{\sum_{i=1}^\infty\diam(U_i)^s\colon \bigcup_{i=1}^\infty U_i\supset S, \diam U_i<\delta}
 $$ 
 to be the \emph{$s$-dimensional Hausdorff outer measure} of $S$. It is well known that $\mathcal{H}^s$ restricted to Borel subsets of $X$ is a measure. We define the \emph{Hausdorff dimension} of $S$ as $$\dim_H(S)=\inf\{s\geq 0: \mathcal{H}^s(S)=0\}=\sup\{s\geq 0:\; \mathcal{H}^s(S)=\infty\}.$$
\end{definition}
The following theorem is a compilation of known properties of the Hausdorff dimension. We will use these facts without further comments.
\begin{theorem}
    Let $(X, d)$ be a metric space.
    \begin{enumerate}
        \item If $E\subset F\subset X$, then $\dim_H(E)\leq \dim_H(F)$,
        \item if $F_1,F_2,\dots\subset X$ is a sequence of subsets of $X$, then $\dim_H(\bigcup_{n=1}^\infty F_n)=\sup_{n\in\mathbb{N}}\{\dim_H(F_n)\}$,
        \item if $E\subset X$ is countable, then $\dim_H(E)=0$.
    \end{enumerate}
\end{theorem}

The following two results will be crucial in our reasoning.

\begin{theorem}[{\cite[Lemma 1, (4.16)]{Hoeffding}}]\label{Hoeffding_lemma}
    Let $X$ be a random variable on a probability space $(X, \mu)$ such that $a\leq X\leq b$ for $a, b\in\mathbb{R}$. Then
    \begin{equation*}
        \mathbb{E}e^{\lambda(X-\mathbb{E}X)}\leq e^{\frac{\lambda^2}{8}(b-a)^2}.
    \end{equation*}
\end{theorem}
\begin{theorem}[{\cite[Theorem 1]{Chernoff1952}}]\label{Chernoff}
    Let $X$ be a random variable on a probability space $(X, \mu)$. Then
    \begin{equation*}
        \mu(\{X\geq t\})\leq \inf_{\lambda>0}\mathbb{E}(e^{\lambda X})e^{-\lambda t}.
    \end{equation*}
\end{theorem}

Let us recall one more classical result for limit convergence of natural sequences.
\begin{theorem}[Stolz-Cesaro Theorem]
    Let $(a_n)_{n=1}^\infty, (b_n)_{n=1}^\infty$ be two strictly monotonic and divergent sequences. Then we have that
    \begin{equation*}
        \lim_{n\to\infty}\frac{a_n-a_{n-1}}{b_n-b_{n-1}}=\ell\implies \lim_{n\to\infty}\frac{a_n}{b_n}=\ell,
    \end{equation*}
    where $\ell\in\mathbb{R}\cup\{-\infty, \infty\}$.
\end{theorem}

\subsection{Hausdorff dimension of Moran sets}\label{Moran}
The following theorem is due to Feng, Wen, and Wu \cite{moran_technique}. We present here the adapted version that suits best our considerations.

Let $(d_n)_{n=1}^\infty\subset\mathbb{N}_{\geq 2}$ and $(c_n)_{n=1}^\infty\subset (0, 1)$ be two sequences such that $d_n c_n\leq 1$ for all $n\in\mathbb{N}$. Let
\begin{equation*}
    D_n=\{s_1\dots s_n\colon\; 1\leq s_j\leq d_j,\;j=1,\dots, n\}
\end{equation*}
be the set of all finite blocks of length $n\in\mathbb{N}$ such that each symbol $s_j\in\{1,\dots, d_j\}$. Define $D=\bigcup_{n\geq 0}D_n$, where we assume that $D_0=\emptyset$.
\begin{definition}
    Let $\mathcal{F}=\{J_\omega\colon\;\omega\in D\}$ be a collection of closed subsets of the unit interval $[0, 1]$ with the following properties:
    \begin{enumerate}
        \item $J_\emptyset=[0, 1]$,
        \item For any $n\in\mathbb{N}$ and $\omega$ in $D_{n-1}$, the sets $J_{\omega 1},\dots, J_{\omega d_n}$ are subintervals of $J_{\omega}$ with pairwise disjoint interiors,
        \item For any $n\in\mathbb{N}$ and $\omega$ in $D_{n-1}$ we have $\frac{\diam(J_{\omega s})}{\diam J_{\omega}}=c_n$, where $s\in\{1,\dots, d_n\}$.
    \end{enumerate}
    Then the set $\mathbb{C}_\infty=\bigcap_{n\geq 1}\bigcup_{\omega\in D_n}J_\omega$ is a {\it homogeneous Moran set determined by $\mathcal{F}$.}
\end{definition}
\begin{theorem}[{\cite[Theorem 2.1]{moran_technique}}]
    The Hausdorff dimension
    \begin{equation*}
        \dim_H(\mathbb{C}_\infty)\geq \liminf_{n\to\infty}\frac{\log(d_1\dots d_n)}{-\log (c_1\dots c_{n+1}d_{n+1})}.
    \end{equation*}
\end{theorem}

\section{Proof of Theorem~\ref{Main}}
We first observe the following straightforward corollary of Theorem~\ref{BesicovitchTypical}.
\begin{corollary}
    The set $\textit{TpcN}\setminus\textit{Normal}$ has full Hausdorff dimension.
\end{corollary}
\begin{proof}
    Observe that we have $\textit{Normal}\subset B\pr{\frac{1}{2}}$ and hence $$\textit{TpcN}\setminus B\pr{\frac{1}{2}}\subset\textit{TpcN}\setminus\textit{Normal}.$$
    Thus, we have that $B(\alpha)\cap E(1)\subset\textit{TpcN}\setminus\textit{Normal}$ for all ${\alpha\in(0, 1)\setminus\{\frac{1}{2}\}}$. Therefore, we may use Theorem~\ref{BesicovitchTypical} to show that
    \begin{equation*}
        \begin{aligned}
        \dim_H(\textit{TpcN}\setminus\textit{Normal})&\geq \dim_H\pr{\bigcup_{\alpha\in\mathbb{Q}\cap{[0, 1]}\setminus\{\frac{1}{2}\}}B(\alpha)\cap E(1)}\\
            &=\sup\br{\dim_H(B(\alpha)\cap E(1))\colon \; {\alpha\in\mathbb{Q}\cap{[0, 1]}\setminus\br{1/2}}}=1.\qedhere
        \end{aligned}
    \end{equation*}
\end{proof}

We may now focus our attention on showing that $\textit{Normal}\setminus \textit{TpcN}$ has full Hausdorff dimension. This proof is considerably more involved. 

Given sequences $(\delta_k)_{k=1}^\infty \subset (0,1/2]$, $(\epsilon_k)_{k=1}^\infty \subset (0,1)$, and $(N_k)_{k=1}^\infty \subset\N$, we define a special set $X=X((\delta_k),(\epsilon_k),(N_k))$ as follows:
\begin{equation*}
    X=\{x\in[0, 1]\colon\; x=0.B_1\beta_1B_2\beta_2\dots,\; B_k\in L_k,\; \beta_k=1^{(\floor{4\log N_k}+1)}, \;k\in\mathbb{N}\},
\end{equation*}
where $L_k$ is the family of all $(\epsilon_k, k)$-normal blocks of length $N_k$,
each $B_k\in L_k$, and $\beta_k=1^{(\floor{4\log N_k}+1)}$ is a block of length $\floor{4\log N_k}+1$ of consecutive ones. Under the conditions introduced below, we will prove in Lemma \ref{X_Moran} that $X$ is a homogeneous Moran set, which will allow us to estimate its Hausdorff dimension using the techniques developed in \cite{moran_technique}.

Although the set $X$ is defined for arbitrary sequences as above, we must
impose additional conditions on $(\delta_k)$, $(\epsilon_k)$, and $(N_k)$
to ensure that the resulting set has the properties required in the proof.
We collect these conditions in the following definition.
\begin{definition}
    A triple $((\delta_k),(\epsilon_k),(N_k))$ is called \textit{good} if the
following conditions hold:
\begin{enumerate}
    \item The product $\prod_{k=1}^\infty(1-\delta_k)$ converges,
    \item for all $k \in \N$, $\epsilon_{k+1}\leq\frac{1}{3}\epsilon_k$
    \item  $N_k>\max\{2k, \floor{4\log N_k}+1\}$,
    \item $\#L_k > (1-\delta_k)2^{N_k}$,
    \item There exists $s \in [1,\infty)$ such that for all $k \in \N$ we have $N_k<N_{k+1}<sN_k$
\end{enumerate}
\end{definition}

We will show the following theorem.
\begin{theorem}\label{FractalHelps}
If $((\delta_k),(\epsilon_k),(N_k))$ is good, then $X \subset \textit{Normal}\setminus \textit{TpcN}$ and $X$ has full Hausdorff dimension.
\end{theorem}
Theorem~\ref{Main} will thus immediately follow from Theorem~\ref{FractalHelps} provided that we show that there exist any good triples. Thus, we will first show that there exist good triples by constructing an example.

\begin{lemma}
    The set of good triples is nonempty. In particular, we may choose $\delta_k=\frac{1}{(k+1)^2}$, $\epsilon_k=3^{-k}$, and $N_k=k^23^{2k+1}+k-1$.
\end{lemma}

\begin{proof}

Define the sequence $(\delta_k)$ by $\delta_k=\frac{1}{(k+1)^2} \in (0,1/2]$, so the product $\prod_{k=1}^\infty (1-\delta_k)=\frac{1}{2}$. Furthermore, let $\epsilon_k=\frac{1}{3^k}$. It remains to define $(N_k)_{k=1}^\infty$. In order to do so, we need to estimate the minimal value $N_k$ for which the number of $(\epsilon_k, k)$-normal blocks is bigger than $(1-\delta_k)2^{N_k}$, i.e. we need to satisfy property (4) of good triples.

Let $Y=\{0, 1\}^\mathbb{N}$ be the space of all infinite binary sequences. Let $(Y, \mu)$ be a probability space, where we consider $\mu$ to be the standard equidistributed Bernoulli measure on $\{0, 1\}$. One can see that we can express the occurrences of given block $\omega$ in $y$ by binary random variables
\begin{equation*}
    X_i(\omega)=\left\{
    \begin{aligned}
        &1, \; y_i\dots y_{i+k-1}=\omega,\\
        &0, \;\text{otherwise}.
    \end{aligned}
    \right.
\end{equation*}
Let $M_k=k\cdot m(k)$  for some positive integer valued function $m(k)$ that we will determine later. We also set $N_k=M_k+(k-1)$.
We are interested in minimizing the value
\begin{equation*}
    \mu\pr{\br{\frac{1}{M_k}|S_{M_k}(\omega)-\mathbb{E}S_{M_k}(\omega)|\geq \epsilon_k}},
\end{equation*}
where $S_n(\omega)=\sum_{i=1}^{n} X_i(\omega)$. In general, the sum $S_n(\omega)$ is not composed of independent random variables. In fact, one can see that the value $X_i(\omega)$ depends on $k$ consecutive digits of $y$. Thus, we have that the variables $X_i(\omega), X_{i+k}(\omega)$ are necessarily independent.  Note that the sum $S_n$ acts on the first $n+(k-1)$ digits in $y$.

It remains to construct an appropriate sequence $(M_k)_{k=1}^\infty$ so that the triple $((\delta_k),(\epsilon_k),(N_k))$ is good. Note that
\begin{equation*}
    S_{M_k}(\omega)=\sum_{i=1}^{k}\sum_{j=0}^{m(k)-1}X_{i+kj}(\omega)=\sum_{i=1}^{k}\hat{X_i}(\omega),
\end{equation*}
where each $\hat{X}_i(\omega)=\sum_{j=0}^{m(k)-1}X_{i+kj}(\omega)$ is a sum of $m(k)$ independent random variables.

We now turn towards the moment generating function $e^{S_{M_k}(\omega)}$ of $S_{M_k}(\omega)$. Since our reasoning will be symmetric, we present it here only for $S_{M_k}(\omega)-\mathbb{E}S_{M_k}(\omega)$. We have that
\begin{equation*}
    \mathbb{E}e^{\lambda(S_{M_k}(\omega)-\mathbb{E}S_{M_k}(\omega))}=\mathbb{E}e^{\lambda(\sum_{i=1}^{k}(\hat{X}_i(\omega)-\mathbb{E}\hat{X}_i(\omega)))}=\mathbb{E}e^{\sum_{i=1}^{k}\frac{1}{k}\lambda k(\hat{X}_i(\omega)-\mathbb{E}\hat{X}_i(\omega))}.
\end{equation*}
Since $e^x$ is convex, using Jensen's inequality and the fact that the expected value is linear, see that
\begin{equation*}
    \mathbb{E}e^{\lambda(S_{M_k}(\omega)-\mathbb{E}S_{M_k}(\omega))}\leq \frac{1}{k}\sum_{i=1}^{k}\mathbb{E}e^{\lambda k(\hat{X}_i(\omega)-\mathbb{E}\hat{X}_i(\omega))}.
\end{equation*}
Recall that each $\hat{X}_i(\omega)$ is a sum of $m(k)$ independent binary random variables. Hence, using Theorem \ref{Hoeffding_lemma} we obtain
\begin{equation*}
    \begin{aligned}
        \frac{1}{k}\sum_{i=1}^{k}\mathbb{E}e^{\lambda k(\hat{X}_i(\omega)-\mathbb{E}\hat{X}_i(\omega))}&=\frac{1}{k}\sum_{i=1}^{k}\prod_{j=0}^{m(k)-1}\mathbb{E}e^{\lambda k(X_{i+kj}(\omega)-\mathbb{E}X_{i+kj}(\omega))}\\
        &\leq \frac{1}{k}\sum_{i=1}^{k} e^{\frac{(\lambda k)^2}{8}m(k)}=e^{\frac{\lambda^2 k M_k}{8}}.
    \end{aligned}
\end{equation*}
We now use Theorem \ref{Chernoff}. We have

\begin{equation*}
    \mu(\{S_{M_k}(\omega)-\mathbb{E}S_{M_k}(\omega)\geq \epsilon_k\})\leq \inf_{\lambda>0}\mathbb{E}e^{\lambda (S_{M_k}(\omega)-\mathbb{E}S_{M_k}(\omega))}e^{-\lambda \epsilon_k}\leq \inf_{\lambda>0}e^{-\lambda \epsilon_k+\frac{\lambda^2 k M_k}{8}}.
\end{equation*}
We minimize the last expression for $\lambda$ and get that $\lambda_\text{min}=\frac{4\epsilon_k}{k M_k}>0$. Substituting $\lambda=\lambda_\text{min}$ in the above equation we get
\begin{equation*}
    \mu(\{S_{M_k}(\omega)-\mathbb{E}S_{M_k}(\omega)\geq\epsilon_k\})\leq e^{-\frac{2\epsilon_k^2}{k M_k}}.
\end{equation*}
As mentioned before, our reasoning is symmetric and thus
\begin{equation*}
    \mu(\{|S_{M_k}(\omega)-\mathbb{E}S_{M_k}(\omega)|\geq\epsilon_k\})\leq 2e^{-\frac{2\epsilon_k^2}{k M_k}}.
\end{equation*}
Hence, we have that
\begin{equation*}
    \begin{aligned}
        &\mu\pr{\br{\frac{1}{M_k}|S_{M_k}(\omega)-\mathbb{E}S_{M_k}(\omega)|\geq \epsilon_k}}=\\
        =&\mu(\{|S_{M_k}(\omega)-\mathbb{E}S_{M_k}(\omega)|\geq M_k\epsilon_k\})\leq 2e^{-2 m(k)\epsilon_k^2}.
    \end{aligned}
\end{equation*}
Moreover, one can see that $\mathbb{E}S_{M_k}=\frac{M_k}{2^k}$, as the underlying measure is the standard equidistributed Bernoulli measure on $\{0, 1\}$. Hence, we have that
\begin{equation*}
    \mu\pr{\br{\abs{\frac{S_{M_k}}{M_k}-\frac{1}{2^k}}\geq \epsilon_k}}\leq 2e^{-2 m(k)\epsilon_k^2}.
\end{equation*}
Since we need to take into account all blocks $\omega\in\{0, 1\}^k$ of length $k$, one can see that the value
\begin{equation*}
    \mu(\{y_1\dots y_{M_k+(k-1)}\text{ is not }(\epsilon_k, k)\text{-normal}\})\leq 2^{k+1}e^{-2 m(k)\epsilon_k^2},
\end{equation*}
where $y_1\dots y_{M_k+(k-1)}$ is the initial block of length $M_k+(k-1)$ in $y$. Hence, the number of blocks of length $M_k+(k-1)$ which are not $(\epsilon_k, k)$-normal is bounded above by $2^{(M_k+(k-1))+\alpha_k}$, where $$\alpha_k=k+1-\frac{2 m(k)\epsilon_k^2}{\ln 2}.$$

Therefore, in order to establish the lower bound for the value $M_k$ to satisfy property (4) of good triples, we need to solve the following inequality:
\begin{equation*}
    2^{M_k+k-1}-2^{M_k+k-1+\alpha_k}>(1-\delta_k)2^{M_k+k-1}\iff 2^{\alpha_k}<\delta_k.
\end{equation*}
We conclude that
\begin{equation*}
    m(k)>\frac{\ln 2}{2\epsilon_k^2}(k+1+2\log(k+1))=\frac{9^k\ln 2}{2}(k+1 +2\log(k+1)).
\end{equation*}
Thus, it is enough to take $m(k)=k\cdot 3^{2k+1}$ and thus $M_k=k^2\cdot 3^{2k+1}$. Observe that this sequence satisfies our initial assumption $M_k=k\cdot m(k)$. As mentioned before, we define the sequence $(N_k)_{k=1}^\infty$ by setting $N_k=M_k+(k-1)$. One can see that the limit
\begin{equation*}
    \lim_{k\to\infty}\frac{\log N_{k+1}}{N_k}=0
\end{equation*}
and moreover, the sequence $(N_k)_{k=1}^\infty$ satisfies
\begin{enumerate}
    \item $N_k>2k$,
    \item $N_k>\floor{4\log N_k}+1$,
    \item $N_k<N_{k+1}<37 N_{k}$,
    \item the number of $(\epsilon_k, k)$-normal blocks of length $N_k$ is bigger than $(1-\delta)2^{N_k}$.
\end{enumerate}
Hence, the triple $((\delta_k)_{k=1}^\infty, (\epsilon_k)_{k=1}^\infty, (N_k)_{k=1}^\infty)$ is good.
\end{proof}

Now that we know that the set of good triples in nonempty, we will prove Theorem~\ref{FractalHelps}. For the rest of this section, assume that $((\delta_k),(\epsilon_k),(N_k))$ is good and $X=X((\delta_k),(\epsilon_k),(N_k))$. We will present our reasoning as a sequence of lemmas, which combined show our main assertion.

\begin{lemma}
    $X \cap \textit{TpcN}=\emptyset$
\end{lemma}
\begin{proof}
    For $x \in X$, let us evaluate the fraction $\frac{R_n(x)}{\log n}$ at positions $m_k=\sum_{i=1}^k |B_i|+|\beta_i|$. By property $(3)$ of good triples, we have that
    \begin{equation*}
    \begin{aligned}
        \log m_k&=\log (\sum_{i=1}^k |B_i|+|\beta_i|)\\
        &<\log(2\sum_{i=1}^k |B_i|)<\log (2kN_k)<2\log N_k
    \end{aligned}
    \end{equation*}
    and up to the position $m_k$ we have a sequence of consecutive ones of length $\floor{4\log N_k}+1$. Hence, there exists a subsequence $(m_k)_{k=1}^\infty$ along which the value $\liminf_{k\to\infty}\frac{R_{m_k}(x)}{\log m_k}\geq 2$. Thus, $x \notin \textit{TpcN}$.
\end{proof}

\begin{lemma}\label{m_0}
    Let $k \in \N$. Then the block $B_{k+m}$ is $(\tilde{\epsilon}_m,k)$ normal where $\tilde{\epsilon}_m=o(1)$.
\end{lemma}
\begin{proof}
    We note that by Lemma~\ref{epsilonkmnormality}, the block $B_{k+m}$ is $\pr{2^m\epsilon_{k+m}+\frac{m}{N_{k+m}},k}$-normal. However, $\epsilon_{k+m} \leq 3^{-m}\epsilon_k$ and $\lim_{m \to \infty} \frac{k+m}{N_{k+m}}=0$, so the block $B_{k+m}$ is $((2/3)^m \epsilon_k+o(1),k)$-normal.
\end{proof}

\begin{lemma}
    $X \subset \textit{Normal}$.
\end{lemma}
\begin{proof}

    Let $x \in X$.
    We will prove that $x$ is normal by using the Hotspot lemma (Theorem \ref{constant_C}). Fix $k\in\mathbb{N}$ and $\epsilon<\frac{1}{2^k}$, and let $\omega\in\{0, 1\}^k$ be a finite binary sequence of length $k$. We will estimate the value $\limsup_{n\to\infty}\frac{N(x, \omega, n)}{n}$.
   It follows from Lemma \ref{m_0} that there exists $m_0$ such that all blocks $B_{k+m}$ in the binary expansion of $x$  are $(\epsilon, k)$-normal for $m\geq m_0$. Let $m(n)$ be the maximal value for which $\sum_{i=1}^{m(n)}(N_i+\floor{4\log N_i}+1)< n$.  We have that
    \begin{equation*}
        \begin{aligned}
            \frac{N(x,\omega, n)}{n}&\leq \frac{\sum_{i=1}^{k+m_0-1}N(B_i, \omega, N_i)}{n}+\frac{\sum_{i=1}^{k+m_0-1}k+\floor{4\log N_i}+1}{n}\\
            &+\frac{\sum_{i=k+m_0}^{m(n)}N(B_i, \omega, N_i)}{\sum_{i=k+m_0}^{m(n)}N_i}+\frac{\sum_{i=k+m_0}^{m(n)}k+\floor{4\log N_i}+1}{\sum_{i=k+m_0}^{m(n)}N_i}\\
            &+\frac{N(B_{m(n)+1},\omega, N_{m(n)+1})+k+\floor{4\log N_{m(n)+1}}+1}{N_{m(n)}}.
        \end{aligned}
    \end{equation*}
    One can see that
    \begin{equation*}
        \begin{aligned}
            &\lim_{n\to\infty}\frac{\sum_{i=1}^{k+m_0-1}N(B_i, \omega, N_i)}{n}+\frac{\sum_{i=1}^{k+m_0-1}k+\floor{4\log N_i}+1}{n}\\
            &+\frac{\sum_{i=k+m_0}^{m(n)}k+\floor{4\log N_i}+1}{\sum_{i=k+m_0}^{m(n)}N_i}+\frac{k+\floor{4\log N_{m(n)+1}}+1}{N_{m(n)}}=0,
        \end{aligned}
    \end{equation*}
    which is a straightforward consequence of the Stolz-Cesaro Theorem. Thus, the value
    \begin{equation*}
        \begin{aligned}
            \limsup_{n\to\infty}\frac{N(x, \omega, n)}{n}\leq \limsup_{n\to\infty}&\frac{\sum_{i=k+m_0}^{m(n)}N(B_i, \omega, N_i)}{\sum_{i=k+m_0}^{m(n)}N_i}\\
            +&\frac{N(B_{m(n)+1},\omega, N_{m(n)+1})}{N_{m(n)}}.
        \end{aligned}
    \end{equation*}
    Let us compute each term separately. We have that 
    \begin{equation*}
        \limsup_{n\to\infty}\frac{\sum_{i=k+m_0}^{m(n)}N(B_i, \omega, N_i)}{\sum_{i=k+m_0}^{m(n)}N_i}\leq \limsup_{n\to\infty}\frac{\sum_{i=k+m_0}^{m(n)}(\epsilon+\frac{1}{2^k})N_i}{\sum_{i=k+m_0}^{m(n)}N_i}\leq \frac{2}{2^k}.
    \end{equation*}
    Similarly, we have that
    \begin{equation*}
        \limsup_{n\to\infty}\frac{N(B_{m(n)+1}\,\omega, N_{m(n)+1})}{N_{m(n)}}\leq \pr{\epsilon+\frac{1}{2^k}}\frac{N_{m(n)+1}}{N_{m(n)}}\leq \frac{2s}{2^k},
    \end{equation*}
    since we have that $N_{i+1} < sN_i$ for all $i\in\mathbb{N}$. Thus,
    \begin{equation*}
        \limsup_{n\to\infty}\frac{N(x, \omega, n)}{n}\leq \frac{2(s+1)}{2^k}.
    \end{equation*}
    Hence, the number $x$ is normal.
\end{proof}

\begin{lemma}\label{X_Moran}
    The set $X$ is a homogeneous Moran set.
\end{lemma}

\begin{proof}
    Let $(d_n)_{n=1}^\infty$ be a sequence defined by setting $d_n=\# L_n$ and $(c_n)_{n=1}^\infty$ be a sequence defined by setting $c_n=2^{-(N_n+\floor{4\log N_n}+1)}$. One can see that these sequences satisfy the properties stated in Section \ref{Moran}. Indeed, observe since $N_k>\floor{4\log N_k}+1$ and $\delta_k\leq \frac{1}{2}$ for all $k\in\mathbb{N}$, we have that the value $(1-\delta_k)2^{N_k}\geq 2$.

Thus, in each $L_k$ there are at least 2 distinct binary sequences. Moreover, one can see that there are at most $2^{N_k}$ blocks in each $L_k$ and hence, $d_nc_n\leq 1$ for all $n\in\mathbb{N}$.
    
    We may identify the set $D_n$ defined in Section \ref{Moran} with the set $$\hat{L}_n=\{\hat{B}=B_1\beta_1\dots B_n\beta_n\colon\; B_i\in L_i, \; i=1,\dots, n\}.$$ We do so by enumerating elements in each $L_n$, i.e. we associate a natural number $M(B_n)$ to each $B_n\in L_n$. Then we set $M(\hat{B})=M(B_1)\dots M(B_n)$, which is the concatenation of numbers associated to the sequences $B_1,\dots, B_n$ that form $\hat{B}$. Observe that setting $\hat{L}=\bigcup_{n\geq 0}\hat{L}_n$ we further extend this identification with the set $D=\bigcup_{n\geq 0}D_n$, where we assume that $\hat{L}_0=\emptyset$. In this setting, the family $\mathcal{F}$ is the family of all cylinder sets $J_{M(\hat{B})}=[\hat{B}]\subset [0, 1]$ defined by the sequences $\hat{B}\in\hat{L}$, i.e.
    $$[\hat{B}]=\{x\in[0, 1]\colon \;x_i=\hat{B}_i, \;i=1,\dots, |\hat{B}|\}$$
    and $J_\emptyset=[0, 1]$. Thus, one can see that the sets $J_{M(\hat{B})M(B_{n+1})}, J_{M(\hat{B})B'_{n+1}}$ for distinct $B_{n+1}, B'_{n+1}$ are two subintervals of $J_{M(\hat{B})}$ with disjoint interiors, as these are two different cylinder sets in $[0, 1]$. Moreover, we have that $$\frac{\diam\pr{J_{M(\hat{B})M(B_{n+1})}}} {\diam\pr{J_{M(\hat{B})}}}=c_{n+1},\text{ where }B_{n+1}\in L_{n+1}, \;\hat{B}\in\hat{L}_n.$$ Thus, we have that
    \begin{equation*}
        \mathbb{C}_\infty=\bigcap_{n\geq 1}\bigcup_{M(\hat{B})\in D_n} J_{M(\hat{B})}=\bigcap_{n\geq 1}\bigcup_{\hat{B}\in\hat{L}_n}[\hat{B}]=X.
    \end{equation*}
\end{proof}

\begin{lemma}
    The Hausdorff dimension $\dim_H(X)=1$.
\end{lemma}
\begin{proof}
    We have that 
    \begin{equation*}
        \dim_H(X)\geq \liminf_{n\to\infty}\frac{\log(d_1\dots d_n)}{-\log(c_1\dots c_{n+1}d_{n+1})},
    \end{equation*}
    where each $d_k>(1-\delta_k)2^{N_k}$, $c_k=2^{-(N_k+\floor{4\log N_k}+1)}$. Then
    \begin{equation*}
        \begin{aligned}
            \dim_H(X)&\geq \\&\geq \liminf_{n\to\infty}\frac{\log(\prod_{k=1}^n (1-\delta_k)2^{N_k})}{\log(\prod_{k=1}^{n+1} 2^{N_k+\floor{4 \log N_k}+1})-\log 2^{N_{k+1}}(1-\delta_{k+1})}\\
            &\geq\liminf_{n\to\infty}\frac{\log(\prod_{k=1}^n 2^{N_k})+\log(\prod_{k=1}^n(1-\delta_k))}{\log(\prod_{k=1}^n 2^{N_k+\floor{4\log N_k}+1})+\log(2^{\floor{4\log N_{n+1}}+1})}\\
            &=\liminf_{n\to\infty}\frac{N_n}{N_n+\log N_n+\log N_{n+1}},
        \end{aligned}
    \end{equation*}
    which follows from the Stolz-Cesaro Theorem applied to the simplified fraction. Thus, we have that
    \begin{equation*}
        \dim_H(X)\geq \liminf_{n\to\infty}\frac{N_n}{N_n(1+\frac{\log N_n}{N_n}+\frac{\log N_{n+1}}{N_n})}=1.\qedhere
    \end{equation*}
\end{proof}
The proof of Theorem~\ref{FractalHelps} is complete.

\section{Ackowledgements}
B.~Mance and J.~Tomaszewski are supported by grant 2019/34/E/ST1/00082 for the project ``Set theoretic methods in dynamics and number theory,'' NCN (The National Science Centre of Poland).

J.~Tomaszewski gratefully acknowledges the hospitality of the University of Maryland, College Park, during a research visit in which part of this work was carried out. The visit was supported by the Fulbright Program.

\printbibliography
\end{document}